\documentclass{amsart}
\usepackage{amsmath}
\usepackage{amssymb}
\usepackage{enumerate}
\usepackage{mathrsfs}
\usepackage{color}
\usepackage{tikz}
\usepackage{extarrows}

\newtheorem{theorem}{Theorem}[section]
\newtheorem{lemma}[theorem]{Lemma}
\newtheorem{definition}[theorem]{Definition}
\newtheorem{example}[theorem]{Example}

\newtheorem{proposition}[theorem]{Proposition}

\newtheorem{remark}[theorem]{Remark}

\numberwithin{equation}{section}

\def\Z{{\mathbb{Z}}}
\def\Q{{\mathbb{Q}}}
\def\R{{\mathbb{R}}}
\def\C{{\mathbb{C}}}
\def\N{{\mathbb{N}}}

\def\x{{\mathbf{x}}}

\def\a{{\boldsymbol{\alpha}}}
\def\C{{\boldsymbol{\gamma}}}

\def\b{{\boldsymbol{\beta}}}

\def\be{{\boldsymbol{e}}}

\def\A{{\mathscr{A}}}
\def\B{{\mathscr{B}}}

\def\supp{\hbox{\rm{supp}}}

\def\int{\hbox{\rm{int}}}
\def\New{\hbox{\rm{New}}}
\def\Conv{\hbox{\rm{conv}}}

\def\exp{\hbox{\rm{exp}}}
\def\diag{\hbox{\rm{diag}}}

\def\GL{\hbox{\rm{GL}}}

\begin{document}

\title{Systems of Polynomials with At Least One Positive Real Zero}
\thanks{This work was supported partly by NSFC under grants 61732001 and 61532019.}
\author{Jie Wang}
\address{Jie Wang\\ School of Mathematical Sciences, Peking University}
\email{wangjie212@mails.ucas.ac.cn}
\date{\today}

\begin{abstract}
In this paper, we prove several theorems on systems of polynomials with at least one positive real zero based on the theory of conceive polynomials. These theorems provide sufficient conditions for systems of multivariate polynomials admitting at least one positive real zero in terms of their Newton polytopes and combinatorial structure. Moreover, a class of polynomials attaining their global minimums in the first quadrant are given, which is useful in polynomial optimization.
\end{abstract}

\maketitle
\bibliographystyle{amsplain}

\section{Introduction}
An important problem in real algebraic geometry is bounding the numbers of positive real zeros for polynomials or systems of polynomials. In the univariate case, the well-known Descartes' rule of signs gives a nice upper bound for the number of positive real zeros of a polynomial.

{\bf Descartes' rule of signs} Given a univariate real polynomial $f(x)$ such that the terms of $f(x)$ are ordered by descending variable exponents, the number of positive real zeros of $f(x)$ (counted with multiplicity) is bounded from above by the number of sign variations between consecutive nonzero coefficients. Additionally, the difference between these two numbers (the number of positive real zeros and the number of sign variations) is even.

However, no complete multivariate generalization of Descartes' rule of signs for upper bounds of numbers of positive real zeros is known, except for a conjecture proposed by Itenberg and Roy in 1996 (\cite{it}) and subsequently disproven by T.Y. Li in 1998 (\cite{li}). In \cite{mu}, a special case for systems of polynomials with at most one positive real zero was considered through the theory of oriented matroids. Based on this method, a partially multivariate generalization of Descartes' rule of signs for systems of polynomials supported on circuits can be found in \cite{bi,bi2}.

While lower bounds guarantee the existence of positive real zeros which has applications in fields such as polynomial optimization (the existence of global minimizers, \cite{ba,nie,sc}), chemical reaction networks (the existence of positive steady states, \cite{jo,mu}) and so on, there are few results on lower bounds of numbers of positive real zeros for polynomials or systems of polynomials in the literature. In \cite{so}, a lower bound of numbers of real zeros for systems of Wronski polynomials was derived.

In this paper, we concentrate on the case of lower bound one for the number of positive real zeros. More specifically, we prove several theorems on systems of polynomials with at least one positive real zero (Theorem \ref{thm5}, Theorem \ref{thm2} and Theorem \ref{thm7}), based on the theory of conceive polynomials developed in \cite{ba}. These theorems provide sufficient conditions for systems of polynomials admitting at least one positive real zero in terms of their Newton polytopes and combinatorial structure and can be also seen as a first step to the broader problem of bounding the number of positive real zeros from the below for systems of polynomials.

In polynomial optimization problems, the existence of global minimizers of objective polynomials is often formulated as an assumption for some of the algorithmic approaches (\cite{ba,nie,sc}). However, this assumption is very non-trivial. To verify that a given polynomial has this property is a difficult problem. As another contribution of this paper, we give a class of polynomials which attain their global minimums in the first quadrant (Theorem \ref{thm4}).

\section{Preliminaries}
\subsection{Nonnegative polynomials}
Let $\R[\x]=\R[x_1,\ldots,x_n]$ be the ring of real $n$-variate polynomial, and $\N^*=\N\backslash\{0\}$. Let $\R_+$ be the set of positive real numbers. For a finite set $\A\subseteq\N^n$, we denote by $\Conv(\A)$ the convex hull of $\A$, and by $V(\A)$ the vertices of the convex hull of $\A$. We also denote by $V(P)$ the vertex set of a polytope $P$. For a polynomial $f\in\R[\x]$ of the form $f(\x)=\sum_{\a\in \A}c_{\a}\x^{\a}$ with $c_{\a}\in\R, \x^{\a}=x_1^{\alpha_1}\cdots x_n^{\alpha_n}$, the support of $f$ is $\supp(f):=\{\a\in \A\mid c_{\a}\ne0\}$ and the Newton polytope of $f$ is defined as $\New(f):=\Conv(\supp(f))$. For a polytope $P$, we use $P^{\circ}$ to denote the interior of $P$.

We say that a polynomial $f\in\R[\x]$ is {\em nonnegative} over $A$ for $A\subseteq\R^n$, if for all $\x\in A$, $f(\x)\ge0$. Particularly, a polynomial $f\in\R[\x]$ is nonnegative over $\R^n$ is called a {\em nonnegative polynomial}. A nonnegative polynomial must satisfy the following necessary conditions.
\begin{proposition}(\cite[Theorem 3.6]{re})\label{nc-prop2}
Let $\A\subset\N^n$ and $f=\sum_{\a\in \A}c_{\a}\x^{\a}\in\R[\x]$ with $\supp(f)=\A$. Then $f$ is a nonnegative polynomial only if the followings hold:
\begin{enumerate}[(1)]
  \item $V(\A)\subseteq(2\N)^n$;
  \item If $\a\in V(\A)$, then the corresponding coefficient $c_{\a}$ is positive.
\end{enumerate}
\end{proposition}

\subsection{Circuit polynomials}
\begin{definition}
Let $\A\subseteq(2\N)^n$ and $f\in\R[\x]$. Then $f$ is called a {\em circuit polynomial} if it is of the form
\begin{equation*}
f(\x)=\sum_{\a\in\A}c_{\a}\x^{\a}-d\x^{\b},
\end{equation*}
and satisfies:
\begin{enumerate}[(1)]
  \item $\A$ comprises the vertices of a simplex;
  \item $c_{\a}>0$ for $\a\in\A$;
  \item $\b\in\Conv(\A)^{\circ}$.
\end{enumerate}
\end{definition}

For a circuit polynomial $f=\sum_{\a\in\A}c_{\a}\x^{\a}-d\x^{\b}$, assume
\begin{equation*}
\b=\sum_{\a\in\A}\lambda_{\a}\a\textrm{ with } \lambda_{\a}>0 \textrm{ and } \sum_{\a\in\A}\lambda_{\a}=1,
\end{equation*}
and define the corresponding {\em circuit number} as $\Theta_f:=\prod_{\a\in\A}(c_{\a}/\lambda_{\a})^{\lambda_{\a}}$. The nonnegativity of a circuit polynomial $f$ is decided by its circuit number alone.
\begin{theorem}(\cite[Theorem 3.8]{iw})\label{nc-thm1}
Let $f=\sum_{\a\in\A}c_{\a}\x^{\a}-d\x^{\b}\in\R[\x]$ be a circuit polynomial and $\Theta_f$ its circuit number. Then $f$ is a nonnegative polynomial if and only if $\b\notin(2\N)^n$ and $|d|\le\Theta_f$, or $\b\in(2\N)^n$ and $d\le\Theta_f$.
\end{theorem}

For more details about circuit polynomials, the reader is referred to \cite{iw,wang}.

\subsection{Coercive polynomials}
A polynomial $f\in\R[\x]$ is called a {\em coercive polynomial}, if $f(\x)\rightarrow+\infty$ holds whenever $\left\|\x\right\|\rightarrow+\infty$, where $\left\|\cdot\right\|$ denotes some norm on $\R^n$. Obviously the coercivity of $f$ implies the existence of global minimizers of $f$ over $\R^n$. Necessary conditions (\cite[Theorem 2.8]{ba}) and sufficient conditions (\cite[Theorem 3.4]{ba}) for a polynomial to be coercive were given in \cite{ba}.
\begin{theorem}(\cite[Theorem 2.8]{ba})\label{thm6}
Let $f=\sum_{\a\in \A}c_{\a}\x^{\a}\in\R[\x]$ with $\supp(f)=\A$ be a coercive polynomial and $c_{\mathbf{0}}>0$. Then the following three conditions hold:
\begin{enumerate}[(1)]
  \item $V(\A)\subseteq(2\N)^n$;
  \item If $\a\in V(\A)$, then the corresponding coefficient $c_{\a}$ is positive;
  \item For every $i$, $1\le i\le n$, there exists a vector $2k_i\be_i\in V(\A)$ with $k_i\in\N^*$, where $\be_i$ is the standard basis vector of $\R^n$.
\end{enumerate}
\end{theorem}

For $f=\sum_{\a\in \A}c_{\a}\x^{\a}\in\R[\x]$ with $\supp(f)=\A$, let $\mathscr{F}$ be the set of faces of $\New(f)$ which do not contain the origin point $\mathbf{0}$ and let
\begin{equation}\label{eq9}
D:=\bigcup_{F\in\mathscr{F}}(\A\backslash V(\A))\cap F.
\end{equation}
For the later use, we restate Theorem 3.4 in \cite{ba} for our purpose as follows.
\begin{theorem}(\cite[Theorem 3.4]{ba})\label{thm8}
Let $f=\sum_{\a\in \A}c_{\a}\x^{\a}\in\R[\x]$ with $\supp(f)=\A$ and $c_{\mathbf{0}}>0$. Assume that the necessary conditions in Theorem \ref{thm6} are satisfied and $D$ is as (\ref{eq9}). If for every $\a\in D$, $c_{\a}>0$, then $f$ is a coercive polynomial.
\end{theorem}

For more details about coercive polynomials, the reader is referred to \cite{ba}.

\bigskip
Before proving the main results of this paper, we outline the key ideas as follows. For a polynomial satisfying certain conditions, we first apply an invertible linear transformation to its exponent vectors to obtain a coercive polynomial (Lemma \ref{lm1}). Then we prove that there is a one-to-one correspondence between the minimizers of the original polynomial over the first quadrant and those of the transformed polynomial (Lemma \ref{lm2}). Hence the coercivity of the transformed polynomial and the existence of global minimizers for coercive polynomials imply the existence of minimizers for the original polynomial (Lemma \ref{lm3} and Theorem \ref{thm4}). Finally, we prove three theorems on systems of polynomials with at least one positive real zero by virtue of Lemma \ref{lm1}, Lemma \ref{lm2} and Lemma \ref{lm3} (Theorem \ref{thm5}, Theorem \ref{thm2} and Theorem \ref{thm7}).

\section{Polynomials attaining their global minimums}
In this section, we prove some lemmas and a theorem on the existence of global minimizers in $\R_{+}^n$ for a class of polynomials.

Let $\R[\x^{\pm}]$ denote the Laurent polynomial ring and $g(\x)=\sum_{\a}c_{\a}\x^{\a}\in\R[\x^{\pm}]$. For an invertible matrix $T\in \GL_n(\Q)$, the polynomial obtained by applying $T$ to the exponent vectors of $g$ is denoted by $g^{T}=\sum_{\a}c_{\a}\x^{T\a}$.

Let $\Delta$ be a polytope of dimension $d$. For a vertex $\a$ of $\Delta$, if $\a$ is the intersection of precisely $d$ edges, then we say $\Delta$ is {\em simple} at $\a$. Obviously, a polygon is simple at any vertex.
\begin{lemma}\label{lm1}
Suppose $f=c_{\mathbf{0}}+\sum_{\a\in\A}c_{\a}\x^{\a}-\sum_{\b\in\B}d_{\b}\x^{\b}\in\R[\x^{\pm}]$ such that $\dim(\New(f))=n$, $\B\subseteq\New(f)^{\circ}\cap\Z^n$, $\mathbf{0}\in V(\New(f))$ and $\New(f)$ is simple at $\mathbf{0}$. Then there exists $\A_0=\{\a_1,\ldots,\a_n\}\subseteq V(\New(f))$ and $T\in \GL_n(\Q)$ such that
$$f^{T}=c_{\mathbf{0}}+\sum_{i=1}^nc_{\a_i}x_i^{2k_i}+\sum_{\a\in\A\backslash\A_0}c_{\a}\x^{T\a}-\sum_{\b\in\B}d_{\b}\x^{T\b},$$
where $k_i\in\N^*$, $T\a\in(2\N)^n$ for each $\a\in\A\backslash\A_0$ and $T\b\in\New(f^T)^{\circ}\cap\N^n$ for each $\b\in\B$.
\end{lemma}
\begin{proof}
Since $\New(f)$ is simple at $\mathbf{0}$, $\mathbf{0}$ is the intersection of precisely $n$ edges. Let $\A_0=\{\a_1,\ldots,\a_n\}\subseteq V(\New(f))$ be the other extreme points of these $n$ edges. Let $T'\in \GL_n(\Q)$ such that $T'(\a_1,\ldots,\a_n)=\diag(k_1',\ldots,k_n')$, where $k_i'\in\N^*$. Suppose $\mu\in\N^*$ is the least common multiple of the denominators appearing in the coordinates of $T'\a$ and $T'\b$ for $\a\in\A\backslash\A_0$ and $\b\in\B$. Let $T=2\mu T'$. Then $T\a\in(2\Z)^n$ for each $\a\in\A\backslash\A_0$ and $T\b\in\Z^n$ for each $\b\in\B$. Moreover, since affine transformations keep convexity, we have $T\a\in(2\N)^n$ and $T\b\in\New(f^T)^{\circ}\cap\N^n$. Thus $T$ meets the requirement with $k_i=\mu k_i',i=1,\ldots,n$.
\end{proof}

Consider the bijective componentwise exponential map
\begin{equation}
\exp:\R^n\rightarrow\R^n_+,\quad\x=(x_1,\ldots,x_n)\mapsto e^{\x}=(e^{\x_1},\ldots,e^{\x_n}).
\end{equation}
The image of a polynomial $g(\x)=\sum_{\a}c_{\a}\x^{\a}$ under the map exp is $g(e^{\x})=\sum_{\a}c_{\a}e^{\langle\a,\x\rangle}$, where $\langle\a,\x\rangle=\a^{\mathsf{T}}\x$ is the inner product of $\a$ and $\x$. Clearly, the range of $g(\x)$ over $\R_{+}^n$ is same as the range of $g(e^{\x})$ over $\R^n$.
\begin{lemma}\label{lm2}
Let $g(\x)=\sum_{\a}c_{\a}\x^{\a}\in\R[\x^{\pm}]$ and $T\in \GL_n(\Q)$ such that $g^T(\x)\in\R[\x^{\pm}]$. Then the infimums of $g(\x)$ and $g^{T}(\x)$ over $\R_{+}^n$ are the same. Furthermore, the minimizers (and the zeros) of $g(\x)$ and $g^{T}(\x)$ over $\R_{+}^n$ are in a one-to-one correspondence.
\end{lemma}
\begin{proof}
We only need to show that the same conclusions hold for $g(e^{\x)}$ and $g^{T}(e^{\x})$ over $\R^n$, which easily follow from the equalities $$g(e^{\x})=\sum_{\a}c_{\a}e^{\langle\a,\x\rangle}=\sum_{\a}c_{\a}e^{\langle T\a,T^*\x\rangle}=g^{T}(e^{T^*\x})$$
and
$$g^T(e^{\x})=\sum_{\a}c_{\a}e^{\langle T\a,\x\rangle}=\sum_{\a}c_{\a}e^{\langle \a,T^{\mathsf{T}}\x\rangle}=g(e^{T^{\mathsf{T}}\x}),$$
where $T^*=(T^{-1})^{\mathsf{T}}=(T^{\mathsf{T}})^{-1}$ and $\mathsf{T}$ represents the transpose.
\end{proof}

\begin{lemma}\label{lm3}
Suppose $f=c_{\mathbf{0}}+\sum_{\a\in\A}c_{\a}\x^{\a}-\sum_{\b\in\B}d_{\b}\x^{\b}\in\R[\x^{\pm}]$, $c_{\mathbf{0}},c_{\a},d_{\b}>0$ such that $\A\subseteq\Z^n$, $\B\subseteq\New(f)^{\circ}\cap\Z^n$, $\dim(\New(f))=n$, $\mathbf{0}\in V(\New(f))$ and $\New(f)$ is simple at $\mathbf{0}$. Assume that $\sum_{\a\in\A}c_{\a}\x^{\a}-\sum_{\b\in\B}d_{\b}\x^{\b}$ is not nonnegative over $\R_{+}^n$. Then $f$ has a minimizer over $\R_{+}^n$.
\end{lemma}
\begin{proof}
By Lemma \ref{lm1}, there exists $\A_0=\{\a_1,\ldots,\a_n\}\subseteq V(\New(f))$ and $T\in \GL_n(\Q)$ such that $f^{T}=c_{\mathbf{0}}+\sum_{i=1}^nc_{\a_i}x_i^{2k_i}+\sum_{\a\in\A\backslash\A_0}c_{\a}\x^{T\a}-\sum_{\b\in\B}d_{\b}\x^{T\b}\in\R[\x]$, where $T\a\in(2\N)^n$ for each $\a\in\A\backslash\A_0$ and $T\b\in\New(f^T)^{\circ}\cap\N^n$ for each $\b\in\B$. By Theorem \ref{thm8}, $f^{T}$ is a coercive polynomial, and hence has a global minimizer over $\R^n$. Note that $f^T(|\x|)=c_{\mathbf{0}}+\sum_{i=1}^nc_{\a_i}|x_i|^{2k_i}+\sum_{\a\in\A\backslash\A_0}c_{\a}|\x|^{T\a}-\sum_{\b\in\B}d_{\b}|\x|^{T\b}\le f^T(\x)$, where $|\x|=(|x_1|,\ldots,|x_n|)$. So $f^T$ has a global minimizer $\x^*$ in $\R_{\ge0}^n$, where $\R_{\ge0}$ is the set of nonnegative real numbers. Since $f-c_{\mathbf{0}}$ is not nonnegative over $\R_{+}^n$, by Lemma \ref{lm2}, $f^T-c_{\mathbf{0}}$ is not nonnegative over $\R_{+}^n$. It follows that the global minimum of $f^T$ is lower than $c_{\mathbf{0}}$, and since for $\x\in\R_{\ge0}^n\backslash\R_{+}^n$, $f^T(\x)\ge c_{\mathbf{0}}$, we have $\x^*\in\R_{+}^n$. Thus $f^T$ has a minimizer over $\R_{+}^n$ and so does $f$ by Lemma \ref{lm2}.
\end{proof}

\begin{theorem}\label{thm4}
Suppose $f=\sum_{\a\in\A}c_{\a}\x^{\a}-\sum_{\b\in\B}d_{\b}\x^{\b}\in\R[\x]$, $c_{\a},d_{\b}>0$ such that $\A\subseteq(2\N)^n$, $\B\subseteq\New(f)^{\circ}\cap\N^n$, $\dim(\New(f))=n$. Assume that $\Conv(\A\cup\{\mathbf{0}\})$ is simple at $\mathbf{0}$. If $\mathbf{0}$ is not a global minimizer of $f$, then $f$ has a global minimizer in $\R_{+}^n$.
\end{theorem}
\begin{proof}
Since $f(|\x|)=\sum_{\a\in\A}c_{\a}|\x|^{\a}-\sum_{\b\in\B}d_{\b}|\x|^{\b}\le f(\x)$, we only need to search the global minimizers of $f$ in $\R_{\ge0}^n$, or equivalently in $\{\mathbf{0}\}\cup\R_{+}^n$ (since for $\x\in\R_{\ge0}^n\backslash\R_{+}^n$, $f(\x)\ge f(\mathbf{0})$). If $\mathbf{0}\in\A$ and $f-c_{\mathbf{0}}$ is nonnegative over $\R_{+}^n$, then $\mathbf{0}$ is a global minimizer of $f$. If $\mathbf{0}\in\A$ and $f-c_{\mathbf{0}}$ is not nonnegative over $\R_{+}^n$, then by Lemma \ref{lm3}, $f$ has a minimizer over $\R_{+}^n$, which is also a global minimizer. If $\mathbf{0}\notin\A$ and $f$ is nonnegative over $\R_{+}^n$, then $\mathbf{0}$ is a global minimizer of $f$. If $\mathbf{0}\notin\A$ and $f$ is not nonnegative over $\R_{+}^n$, consider the polynomial $f+c$, $c>0$. By Lemma \ref{lm3}, $f+c$ has a minimizer over $\R_{+}^n$. It follows that $f$ has a minimizer over $\R_{+}^n$, which is also a global minimizer.
\end{proof}

\section{Systems of polynomials with at least one positive real zero}
A {\em positive real zero} of a polynomial or a system of polynomials is a real zero with positive coordinates. Note that the positive real zeros of the polynomials $f(x_1,\ldots,x_n)$ and $f(x_1^2,\ldots,x_n^2)$ are in a one-to-one correspondence. Since we only consider positive real zeros in this paper, we can apply the map $x_i\mapsto x_i^2$ $(1\le i\le n)$ and assume that the supports of polynomials in the following are in $(2\N)^n$ if necessary.

\begin{proposition}\label{thm1}
Let $F$ be the following system of polynomial equations
\begin{equation}
\sum_{\a\in\A}c_{\a}(\a-\C)\x^{\a}-\sum_{\b\in\B}d_{\b}(\b-\C)\x^{\b}=\mathbf{0},
\end{equation}
where $\A\subseteq\N^n$, $c_{\a},d_{\b}>0$ and $\C\in V(\Delta)$, $\B\subseteq\Delta^{\circ}\cap\N^n$ with $\Delta=\Conv(\A\cup\{\C\})$. Assume that $\dim(\Delta)=n$, $\Delta$ is simple at $\C$ and $\sum_{\a\in\A}c_{\a}\x^{\a}-\sum_{\b\in\B}d_{\b}\x^{\b}$ is not nonnegative over $\R_{+}^n$. Then $F$ has at least one positive real zero.
\end{proposition}
\begin{proof}
Consider the polynomial $f=d\x^{\C}+\sum_{\a\in\A}c_{\a}\x^{\a}-\sum_{\b\in\B}d_{\b}\x^{\b}$. Let $f'=f/\x^{\C}=d+\sum_{\a\in\A}c_{\a}\x^{\a-\C}-\sum_{\b\in\B}d_{\b}\x^{\b-\C}$. Then by Lemma \ref{lm3}, $f'$ has a minimizer over $\R_{+}^n$. Assume the minimum of $f'$ over $\R_{+}^n$ is $\xi$. Then $f'(\x)-\xi$ is nonnegative over $\R_{+}^n$ and has a positive real zero. It follows that $f-\xi\x^{\C}=(d-\xi)\x^{\C}+\sum_{\a\in\A}c_{\a}\x^{\a}-\sum_{\b\in\B}d_{\b}\x^{\b}$ is nonnegative over $\R_{+}^n$ and has a positive real zero, which implies that the system of $f-\xi\x^{\C}=0$ and $\nabla(f-\xi\x^{\C})=\mathbf{0}$ ($\nabla$ means the gradient) has a positive real zero. Multiplying $f-\xi\x^{\C}=0$ by $\C$ yields
\begin{equation}\label{eq2}
(d-\xi)\C\x^{\C}+\sum_{\a\in\A}c_{\a}\C\x^{\a}-\sum_{\b\in\B}d_{\b}\C\x^{\b}=\mathbf{0}.
\end{equation}
Multiplying the $i$-th equation of $\nabla(f-\xi\x^{\C})=\mathbf{0}$ by $x_i$ yields
\begin{equation}\label{eq1}
(d-\xi)\C\x^{\C}+\sum_{\a\in\A}c_{\a}\a\x^{\a}-\sum_{\b\in\B}d_{\b}\b\x^{\b}=\mathbf{0}.
\end{equation}
From (\ref{eq1})$-$(\ref{eq2}), we obtain
\begin{equation}
\sum_{\a\in\A}c_{\a}(\a-\C)\x^{\a}-\sum_{\b\in\B}d_{\b}(\b-\C)\x^{\b}=\mathbf{0},
\end{equation}
which is exactly $F$. Thus $F$ has a positive real zero.
\end{proof}

\begin{example}
Consider the following system of polynomial equations with $\A=\{\a_1,\a_2,\a_3,\a_4\}=\{(8,0),(0,8),(4,4),(0,0)\}$, $\B=\{\b_1,\b_2\}=\{(1,4),(3,2)\}$ and $\C=(8,8)$:
\begin{equation}\label{ex1}
\begin{cases}
-8y^8-4x^4y^4-8+21xy^4+5x^3y^2=0\\
-8x^8-4x^4y^4-8+12xy^4+6x^3y^2=0
\end{cases}.
\end{equation}
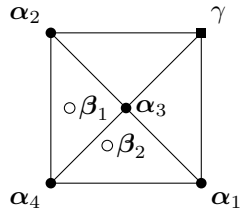
\begin{center}
\begin{tikzpicture}
\draw (0,0)--(0,2);
\draw (0,0)--(2,0);
\draw (2,0)--(2,2);
\draw (0,2)--(2,2);
\draw (0,0)--(2,2);
\draw (0,2)--(2,0);
\fill (0,0) circle (2pt);
\node[below left] (1) at (0,0) {$\a_4$};
\fill (2,0) circle (2pt);
\node[below right] (2) at (2,0) {$\a_1$};
\fill (0,2) circle (2pt);
\node[above left] (3) at (0,2) {$\a_2$};
\fill (1.94,1.94) rectangle (2.06,2.06);
\node[above right] (4) at (2,2) {$\gamma$};
\fill (1,1) circle (2pt);
\node[right] (7) at (1,1) {$\a_3$};
\draw (0.25,1) circle (2pt);
\node[right] (5) at (0.25,1) {$\b_1$};
\draw (0.75,0.5) circle (2pt);
\node[right] (6) at (0.75,0.5) {$\b_2$};
\end{tikzpicture}
\end{center}
The polynomial $x^8+y^8+x^4y^4+1-3xy^4-x^3y^2$ is not nonnegative over $\R_{+}^n$ and hence by Proposition \ref{thm1}, the system (\ref{ex1}) has at least one positive real zero. In other words, the lower bound for the number of positive real zero of (\ref{ex1}) is one. Actually, a computation by {\tt Mathematica} yields two positive real zeros of (\ref{ex1}):
$$(1.13128, 1.23327) \textrm{ and } (0.72571, 0.961524).$$
\end{example}

\begin{lemma}\label{lm5}
Suppose $f_{d}=\sum_{\a\in\A}c_{\a}\x^{\a}+d\x^{\C}-\sum_{\b\in\B}d_{\b}\x^{\b}\in\R[\x]$, $\A\cup\{\C\}\subseteq\N^n$, $c_{\a},d_{\b}>0$ such that $\B\subseteq\Delta^{\circ}\cap\N^n$ with $\Delta=\Conv(\A\cup\{\C\})$. Assume $\dim(\Delta)=n$, $\Delta$ is simple at some vertex $\a_0\in\A$ ($\a_0\ne\C$) and $\sum_{\a\in\A}c_{\a}\x^{\a}-\sum_{\b\in\B}d_{\b}\x^{\b}$ is not nonnegative over $\R_{+}^n$. Let $d^*:=\inf\{d\mid f_d\textrm{ is nonnegative over }\R_{+}^n\}$. Then $f_{d^*}$ has a positive real zero.
\end{lemma}
\begin{proof}
Let $|\B|=l$. For each $\b\in\B$, since $\b\in\Delta^{\circ}$, then there must exist a subset $A_{\b}$ of $\A$ such that $A_{\b}\cup\{\C\}$ comprises the vertices of a simplex $\Delta_{\b}$ containing $\b$ as an interior point. For each $\a\in\cup_{\b\in\B} A_{\b}$, count how many $A_{\b}$'s contain $\a$ and evenly distribute $c_{\a}$. Then we can write
\begin{equation}\label{eq8}
f_d=\sum_{\b\in\B}(\sum_{\a\in A_{\b}}c_{\a\b}\x^{\a}+\frac{d}{l}\x^{\C}-d_{\b}\x^{\b})+\sum_{\a\notin\cup_{\b\in\B} A_{\b}}c_{\a}\x^{\a}
\end{equation}
as a sum of circuit polynomials. Observe that if $d$ is sufficiently large, then every circuit polynomial appearing in (\ref{eq8}) is nonnegative by Theorem \ref{nc-thm1} and hence $f_d$ is nonnegative over $\R_{+}^n$. So the set in the definition of $d^*$ is nonempty and obviously has a lower bound $0$. It follows that $d^*$ exists.

Let $f_{d}'=f_{d}/\x^{\a_0}=c_{\a_0}+\sum_{\a\in\A\backslash\{\a_0\}}c_{\a}\x^{\a-\a_0}+d\x^{\C-\a_0}-\sum_{\b\in\B}d_{\b}\x^{\b-\a_0}$. By Lemma \ref{lm1}, there exists $\A_0=\{\a_1,\ldots,\a_n\}\subseteq V(\A)\backslash\{\a_0\}$ and $T\in \GL_n(\Q)$ such that $f_{d}'^{T}=c_{\a_0}+\sum_{i=1}^nc_{\a_i}x_i^{2k_i}+\sum_{\a\in\A\backslash(\A_0\cup\{\a_0\})}c_{\a}$
$\x^{T(\a-\a_0)}+d\x^{T(\C-\a_0)}-\sum_{\b\in\B}d_{\b}\x^{T(\b-\a_0)}\in\R[\x]$, where $T(\a-\a_0),T(\C-\a_0)\in(2\N)^n$, $T(\b-\a_0)\in\New(f_{d}'^{T})^{\circ}\cap\N^n$ (we assume $\C\notin\A_0$ without loss of generality). By Lemma \ref{lm2}, the nonnegativity of $f_d'^T$ over $\R_{+}^n$ is the same as the nonnegativity of $f_d'$ over $\R_{+}^n$, and hence is the same as the nonnegativity of $f_d$ over $\R_{+}^n$. Let $0<d_0<d^*$. By Theorem \ref{thm8}, $f_{d_0}'^{T}$ is a coercive polynomial. Hence there exists $N>0$ such that for $\left\|\x\right\|_{\infty}>N$, $f_{d_0}'^T(\x)>0$. For $d_0<d<d^*$, since $f_{d}'^{T}(\x)-f_{d_0}'^{T}(\x)=(d-d_0)\x^{T(\C-\a_0)}>0$ over $\R_{+}^n$, we have $f_{d}'^{T}(\x)>f_{d_0}'^{T}(\x)$ over $\R_{+}^n$. Thus for $\left\|\x\right\|_{\infty}>N$ and $\x\in\R_{+}^n$, $f_{d}'^T(\x)>0$. By the definition of $d^*$, $f_{d}$ is not nonnegative over $\R_{+}^n$ and so is $f_{d}'^T$. That is to say, there exists $\x_{d}\in\R_{+}^n$ such that $f_{d}'^T(\x_{d})<0$. It follows $\left\|\x_{d}\right\|_{\infty}\le N$. Let $d\rightarrow d^*$. Then we have $f_d'^T(\x_d)-f_{d^*}'^T(\x_d)=(d-d^*)\x_d^{T(\C-\a_0)}\rightarrow0$. Since $f_{d^*}'^T(\x_d)\ge0$ and $f_d'^T(\x_d)<0$, we must have $f_{d^*}'^T(\x_d)\rightarrow0$. Thus the infimum of $f_{d^*}'^T$ over $\R_{+}^n$ is $0$. It follows that $f_{d^*}'^T-c_{\a_0}$ is not nonnegative over $\R_{+}^n$. So by Lemma \ref{lm3}, $f_{d^*}'^T$ has a minimizer over $\R_{+}^n$, which is a positive real zero of $f_{d^*}'^T$. As a consequence, $f_{d^*}'$ also has a positive real zero by Lemma \ref{lm2} and so does $f_{d^*}$.
\end{proof}

\begin{proposition}\label{thm3}
Let $F$ be the following system of polynomial equations
\begin{equation}
\sum_{\a\in\A}c_{\a}(\a-\C)\x^{\a}-\sum_{\b\in\B}d_{\b}(\b-\C)\x^{\b}=\mathbf{0},
\end{equation}
where $\A\cup\{\C\}\subseteq\N^n$, $c_{\a},d_{\b}>0$ and $\B\subseteq\Delta^{\circ}\cap\N^n$ with $\Delta=\Conv(\A\cup\{\C\})$. Assume that $\dim(\Delta)=n$, $\Delta$ is simple at some vertex $\a_0\in\A$ ($\a_0\ne\C$) and $\sum_{\a\in\A}c_{\a}\x^{\a}-\sum_{\b\in\B}d_{\b}\x^{\b}$ is not nonnegative over $\R_{+}^n$. Then $F$ has at least one positive real zero.
\end{proposition}
\begin{proof}
Consider the polynomial $f_d=\sum_{\a\in\A}c_{\a}\x^{\a}+d\x^{\C}-\sum_{\b\in\B}d_{\b}\x^{\b}$. Define $d^*$ as in Lemma \ref{lm5}. Then by Lemma \ref{lm5}, $f_{d^*}$ has a positive real zero which is also a minimizer of $f_{d^*}$ over $\R_{+}^n$. It implies that the system of $f_{d^*}=0$ and $\nabla(f_{d^*})=\mathbf{0}$ has a positive real zero. Multiplying $f_{d^*}=0$ by $\C$ yields
\begin{equation}\label{eq6}
\sum_{\a\in\A}c_{\a}\C\x^{\a}+d^*\C\x^{\C}-\sum_{\b\in\B}d_{\b}\C\x^{\b}=\mathbf{0}.
\end{equation}
Multiplying the $i$-th equation of $\nabla(f_{d^*})=\mathbf{0}$ by $x_i$ yields
\begin{equation}\label{eq5}
\sum_{\a\in\A}c_{\a}\a\x^{\a}+d^*\C\x^{\C}-\sum_{\b\in\B}d_{\b}\b\x^{\b}=\mathbf{0}.
\end{equation}
From (\ref{eq5})$-$(\ref{eq6}), we obtain
\begin{equation}
\sum_{\a\in\A}c_{\a}(\a-\C)\x^{\a}-\sum_{\b\in\B}d_{\b}(\b-\C)\x^{\b}=\mathbf{0},
\end{equation}
which is exactly $F$. Thus $F$ has a positive real zero.
\end{proof}

\begin{example}
Consider the following system of polynomial equations with $\A=\{\a_1,\a_2,\a_3,\a_4\}=\{(8,8),(8,0),(0,8),(0,0)\}$, $\B=\{\b_1,\b_2\}=\{(1,4),(3,2)\}$ and $\C=(4,4)$:
\begin{equation}\label{ex2}
\begin{cases}
4x^8y^8+4x^8-4y^8-4+9xy^4+x^3y^2=0\\
4x^8y^8-4x^8+4y^8-4+2x^3y^2=0
\end{cases}.
\end{equation}
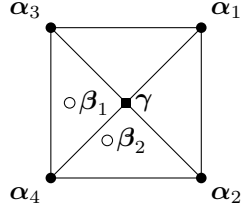
\begin{center}
\begin{tikzpicture}
\draw (0,0)--(0,2);
\draw (0,0)--(2,0);
\draw (2,0)--(2,2);
\draw (0,2)--(2,2);
\draw (0,0)--(2,2);
\draw (0,2)--(2,0);
\fill (0,0) circle (2pt);
\node[below left] (1) at (0,0) {$\a_4$};
\fill (2,0) circle (2pt);
\node[below right] (2) at (2,0) {$\a_2$};
\fill (0,2) circle (2pt);
\node[above left] (3) at (0,2) {$\a_3$};
\fill (2,2) circle (2pt);
\node[above right] (4) at (2,2) {$\a_1$};
\fill (0.94,0.94) rectangle (1.06,1.06);
\node[right] (7) at (1,1) {$\C$};
\draw (0.25,1) circle (2pt);
\node[right] (5) at (0.25,1) {$\b_1$};
\draw (0.75,0.5) circle (2pt);
\node[right] (6) at (0.75,0.5) {$\b_2$};
\end{tikzpicture}
\end{center}
The polynomial $x^8y^8+x^8+y^8+1-3xy^4-x^3y^2$ is not nonnegative over $\R_{+}^n$ and hence by Proposition \ref{thm3}, the system (\ref{ex2}) has at least one positive real zero.
In other words, the lower bound for the number of positive real zero of (\ref{ex2}) is one. Actually, a computation by {\tt Mathematica} yields exactly one positive real zero of (\ref{ex2}): $$(0.752174, 0.974982).$$
\end{example}

Combining Proposition \ref{thm1} with Proposition \ref{thm3}, we obtain the following theorem.
\begin{theorem}\label{thm5}
Let $F$ be the following system of polynomial equations
\begin{equation}
\sum_{\a\in\A}c_{\a}(\a-\C)\x^{\a}-\sum_{\b\in\B}d_{\b}(\b-\C)\x^{\b}=\mathbf{0},
\end{equation}
where $\A\cup\{\C\}\subseteq\N^n$, $c_{\a},d_{\b}>0$ and $\B\subseteq\Delta^{\circ}\cap\N^n$ with $\Delta=\Conv(\A\cup\{\C\})$. Assume that $\dim(\Delta)=n$, $\Delta$ is simple at some vertex and $\sum_{\a\in\A}c_{\a}\x^{\a}-\sum_{\b\in\B}d_{\b}\x^{\b}$ is not nonnegative over $\R_{+}^n$. Then $F$ has at least one positive real zero.
\end{theorem}

\begin{lemma}\label{lm4}
Suppose $f_{d}=\sum_{\a\in\A}c_{\a}\x^{\a}-\sum_{\b\in\B}d_{\b}\x^{\b}-d\x^{\C}\in\R[\x]$, $\A\subseteq\N^n$, $c_{\a},d_{\b}>0$ such that $\B\cup\{\C\}\subseteq\Delta^{\circ}\cap\N^n$ with $\Delta=\Conv(\A)$. Assume that $\dim(\Delta)=n$, $\Delta$ is simple at some vertex $\a_0\in\A$ and $\sum_{\a\in\A}c_{\a}\x^{\a}-\sum_{\b\in\B}d_{\b}\x^{\b}$ is nonnegative over $\R_{+}^n$. Let $d^*:=\sup\{d\mid f_d\textrm{ is nonnegative over }\R_{+}^n\}$. Then $f_{d^*}$ has a positive real zero.
\end{lemma}
\begin{proof}
It is clear that the set in the definition of $d^*$ is nonempty and has upper bounds. Hence $d^*$ exists. Let $f_{d}'=f_{d}/\x^{\a_0}=c_{\a_0}+\sum_{\a\in\A\backslash\{\a_0\}}c_{\a}\x^{\a-\a_0}-\sum_{\b\in\B}d_{\b}\x^{\b-\a_0}-d\x^{\C-\a_0}$. By Lemma \ref{lm1}, there exists $\A_0=\{\a_1,\ldots,\a_n\}\subseteq V(\A)\backslash\{\a_0\}$ and $T\in \GL_n(\Q)$ s.t. $f_{d}'^{T}=c_{\a_0}+\sum_{i=1}^nc_{\a_i}x_i^{2k_i}+\sum_{\a\in\A\backslash(\A_0\cup\{\a_0\})}c_{\a}$
$\x^{T(\a-\a_0)}-\sum_{\b\in\B}d_{\b}\x^{T(\b-\a_0)}-d\x^{T(\C-\a_0)}\in\R[\x]$, where $T(\a-\a_0)\in(2\N)^n$, $T(\b-\a_0),T(\C-\a_0)\in\New(f_{d}'^{T})^{\circ}\cap\N^n$. By Lemma \ref{lm2}, the nonnegativity of $f_d'^T$ over $\R_{+}^n$ is the same as the nonnegativity of $f_d'$ over $\R_{+}^n$, and hence is the same as the nonnegativity of $f_d$ over $\R_{+}^n$. Let $d_0>d^*$. By Theorem \ref{thm8}, $f_{d_0}'^{T}$ is a coercive polynomial. Hence there exists $N>0$ such that for $\left\|\x\right\|_{\infty}>N$, $f_{d_0}'^T(\x)>0$. For $d^*<d<d_0$, since $f_{d}'^{T}(\x)-f_{d_0}'^{T}(\x)=(d_0-d)\x^{T(\C-\a_0)}>0$ over $\R_{+}^n$, we have $f_{d}'^{T}(\x)>f_{d_0}'^{T}(\x)$ over $\R_{+}^n$. Thus for $\left\|\x\right\|_{\infty}>N$ and $\x\in\R_{+}^n$, $f_{d}'^T(\x)>0$. By the definition of $d^*$, $f_d'^T$ is not nonnegative over $\R_{+}^n$. That is to say, there exists $\x_d\in\R_{+}^n$ such that $f_d'^T(\x_d)<0$. It follows $\left\|\x_{d}\right\|_{\infty}\le N$. Let $d\rightarrow d^*$. Then we have $f_d'^T(\x_d)-f_{d^*}'^T(\x_d)=(d^*-d)\x_d^{T(\C-\a_0)}\rightarrow0$. Since $f_{d^*}'^T(\x_d)\ge0$ and $f_d'^T(\x_d)<0$, we must have $f_{d^*}'^T(\x_d)\rightarrow0$. Thus the infimum of $f_{d^*}'^T$ over $\R_{+}^n$ is $0$. It follows that $f_{d^*}'^T-c_{\a_0}$ is not nonnegative over $\R_{+}^n$. So by Lemma \ref{lm3}, $f_{d^*}'^T$ has a minimizer over $\R_{+}^n$, which is a positive real zero of $f_{d^*}'^T$. As a consequence, $f_{d^*}'$ also has a positive real zero by Lemma \ref{lm2} and so does $f_{d^*}$.
\end{proof}

\begin{theorem}\label{thm2}
Let $F$ be the following system of polynomial equations
\begin{equation}
\sum_{\a\in\A}c_{\a}(\a-\C)\x^{\a}-\sum_{\b\in\B}d_{\b}(\b-\C)\x^{\b}=\mathbf{0},
\end{equation}
where $\A\subseteq\N^n$, $c_{\a},d_{\b}>0$ and $\B\cup\{\C\}\subseteq\Delta^{\circ}\cap\N^n$ with $\Delta=\Conv(\A)$. Assume that $\dim(\Delta)=n$, $\Delta$ is simple at some vertex and $\sum_{\a\in\A}c_{\a}\x^{\a}-\sum_{\b\in\B}d_{\b}\x^{\b}$ is nonnegative over $\R_{+}^n$. Then $F$ has at least one positive real zero.
\end{theorem}
\begin{proof}
Consider the polynomial $f_d=\sum_{\a\in\A}c_{\a}\x^{\a}-\sum_{\b\in\B}d_{\b}\x^{\b}-d\x^{\C}$. Define $d^*$ as in Lemma \ref{lm4}. Then by Lemma \ref{lm4}, $f_{d^*}$ has a positive real zero which is also a minimizer of $f_{d^*}$ over $\R_{+}^n$. It implies that the system of $f_{d^*}=0$ and $\nabla(f_{d^*})=\mathbf{0}$ has a positive real zero. Multiplying $f_{d^*}=0$ by $\C$ yields
\begin{equation}\label{eq4}
\sum_{\a\in\A}c_{\a}\C\x^{\a}-\sum_{\b\in\B}d_{\b}\C\x^{\b}-d^*\C\x^{\C}=\mathbf{0}.
\end{equation}
Multiplying the $i$-th equation of $\nabla(f_{d^*})=\mathbf{0}$ by $x_i$ yields
\begin{equation}\label{eq3}
\sum_{\a\in\A}c_{\a}\a\x^{\a}-\sum_{\b\in\B}d_{\b}\b\x^{\b}-d^*\C\x^{\C}=\mathbf{0}.
\end{equation}
From (\ref{eq3})$-$(\ref{eq4}), we obtain
\begin{equation}
\sum_{\a\in\A}c_{\a}(\a-\C)\x^{\a}-\sum_{\b\in\B}d_{\b}(\b-\C)\x^{\b}=\mathbf{0},
\end{equation}
which is exactly $F$. Thus $F$ has a positive real zero.
\end{proof}

\begin{example}
Consider the following system of polynomial equations with $\A=\{\a_1,\a_2,\a_3,\a_4,\a_5\}=\{(8,8),(8,0),(0,8),(4,4),(0,0)\}$, $\B=\{\b\}=\{(3,2)\}$ and $\C=(1,4)$:
\begin{equation}\label{ex3}
\begin{cases}
7x^8y^8+7x^8-y^8+3x^4y^4-1-2x^3y^2=0\\
4x^8y^8-4x^8+4y^8-4+2x^3y^2=0
\end{cases}.
\end{equation}
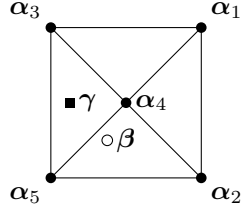
\begin{center}
\begin{tikzpicture}
\draw (0,0)--(0,2);
\draw (0,0)--(2,0);
\draw (2,0)--(2,2);
\draw (0,2)--(2,2);
\draw (0,0)--(2,2);
\draw (0,2)--(2,0);
\fill (0,0) circle (2pt);
\node[below left] (1) at (0,0) {$\a_5$};
\fill (2,0) circle (2pt);
\node[below right] (2) at (2,0) {$\a_2$};
\fill (0,2) circle (2pt);
\node[above left] (3) at (0,2) {$\a_3$};
\fill (2,2) circle (2pt);
\node[above right] (4) at (2,2) {$\a_1$};
\fill (1,1) circle (2pt);
\node[right] (7) at (1,1) {$\a_4$};
\fill (0.19,0.94) rectangle (0.31,1.06);
\node[right] (5) at (0.25,1) {$\C$};
\draw (0.75,0.5) circle (2pt);
\node[right] (6) at (0.75,0.5) {$\b$};
\end{tikzpicture}
\end{center}
The polynomial $x^8y^8+x^8+y^8+x^4y^4+1-x^3y^2$ is nonnegative over $\R_{+}^n$ and hence by Theorem \ref{thm2}, the system (\ref{ex3}) has at least one positive real zero.
In other words, the lower bound for the number of positive real zero of (\ref{ex3}) is one. Actually, a computation by {\tt Mathematica} yields exactly one positive real zero of (\ref{ex3}): $$(0.778814, 0.972957).$$
\end{example}

Combining Theorem \ref{thm5} with Theorem \ref{thm2}, we obtain the following theorem.
\begin{theorem}\label{thm7}
Let $F$ be the following system of polynomial equations
\begin{equation}
\sum_{\a\in\A}c_{\a}(\a-\C)\x^{\a}-\sum_{\b\in\B}d_{\b}(\b-\C)\x^{\b}=\mathbf{0},
\end{equation}
where $\A\subseteq\N^n$, $c_{\a},d_{\b}>0$ and $\B\cup\{\C\}\subseteq\Delta^{\circ}\cap\N^n$ with $\Delta=\Conv(\A)$. Assume that $\dim(\Delta)=n$ and $\Delta$ is simple at some vertex. Then $F$ has at least one positive real zero.
\end{theorem}


\begin{remark}\label{re}
A version of Birch's theorem (\cite{cr,mu}) in statistics states that the following system of polynomial equations
\begin{equation}\label{sec1-eq2}
\sum_{\a\in\A} c_{\a}(\a-\C)\x^{\a}=\mathbf{0},
\end{equation}
where $\A\subseteq\N^n$, $c_{\a}>0$, $\C\in\Conv(\A)^{\circ}\cap\N^n$, $\dim(\Conv(\A))=n$, has exactly one positive real zero. Our theorems hence can be viewed as a partial generalization of Birch's theorem.
\end{remark}

\begin{remark}
In the above theorems, we always assume that the Newton polytope $\Delta$ is simple at some vertex since we need to exploit the property of coercive polynomials in the proofs. It is not clear whether this condition can be dropped.
\end{remark}

As an application, we finally give an example from chemical reaction networks to illustrate our theorems.
\begin{example}
Consider the following reaction network consisting of species $A,B$ and reactions:
\begin{align*}
6A&\xrightarrow{r_1}4A+3B\\
6B&\xrightarrow{r_2}4A+3B\\
3A+5B&\xrightarrow{r_3}4A+3B\\
3A+4B&\xrightarrow{r_4}2A+5B
\end{align*}
with reaction rate constants $r_1,r_2,r_3,r_4>0$. We denote by $x_A,x_B$ the concentrations of the species $A,B$ respectively. Under the assumption of mass-action kinetics, we describe how these concentrations change in time by following system of ODEs:
\begin{equation}\label{ex4}
\begin{cases}
\dot{x}_A=2r_1x^6-4r_2y^6-r_3x^3y^5+r_4x^3y^4\\
\dot{x}_B=-3r_1x^6+3r_2y^6+2r_3x^3y^5-r_4x^3y^4
\end{cases}.
\end{equation}
A {\em positive steady state} of \eqref{ex4} is a concentration-vector $(x^*_A,x^*_B)\in\R_+^2$ at which the right-hand side of the ODEs \eqref{ex4} vanishes.
One can easily check that the system of polynomials in \eqref{ex4} satisfies the conditions of Theorem \ref{thm7} with $\A=\{\a_1,\a_2,\a_3\}=\{(6,0),(0,6),(3,5)\}$, $\B=\{\b\}=\{(3,4)\}$ and $\C=(4,3)$.
\begin{center}
\begin{tikzpicture}
\draw (3,0)--(0,3);
\draw (3,0)--(1.5,2.5);
\draw (1.5,2.5)--(0,3);
\fill (3,0) circle (2pt);
\node[below left] (1) at (3,0) {$\a_1$};
\fill (0,3) circle (2pt);
\node[below left] (2) at (0,3) {$\a_2$};
\fill (1.5,2.5) circle (2pt);
\node[above right] (3) at (1.5,2.5) {$\a_3$};
\fill (1.94,1.44) rectangle (2.06,1.56);
\node[above left] (5) at (2,1.5) {$\C$};
\draw (1.5,2) circle (2pt);
\node[above left] (6) at (1.5,2) {$\b$};
\end{tikzpicture}
\end{center}
Therefore, Theorem \ref{thm7} enables us to give a lower bound, i.e. one, for the number of positive steady states of \eqref{ex4}. Actually, a computation by {\tt Mathematica} yields exactly one positive steady state of \eqref{ex4}: $$(1.75103, 1.49382).$$
\end{example}

\section{Conclusions}
In this paper, sufficient conditions for certain systems of multivariate polynomials admitting at least one positive real zero are given for the first time. These sufficient conditions are expressed in terms of Newton polytopes and their combinatorial structure. It is possible to find applications in polynomial optimization and chemical reaction networks. The further goal is to give upper and lower bounds for the number of positive real zeros to more general systems of multivariate polynomial equations, i.e., a multivariate version of Descartes' rule of signs. We hope the main results of this paper could shed some light on this difficulty problem.

\bibliographystyle{amsplain}

\begin{thebibliography}{99}
\bibitem{ba}
T. Bajbar and O. Stein, Coercive Polynomials and Their Newton Polytopes, {\em SIAM J. Optim.}, 25(3)(2014):1542-1570.

\bibitem{bi}
F. Bihan, Polynomial systems supported on circuits and dessins d'enfant, {\em J. Lond. Math. Soc.}, 75(1)(2007):116-132.

\bibitem{bi2}
F. Bihan and A. Dickenstein, Descartes' Rule of Signs for Polynomial Systems Supported on Circuits, {\em Int. Math. Res. Not.}, 22(2017):6867-6893.

\bibitem{cr}
G. Craciun, A. Dickenstein, A. Shiu, et al., Toric dynamical systems, {\em J. Symbolic Comput.}, 44(11)(2009):1551-1565.

\bibitem{iw}
S. Iliman and T. de Wolff, Amoebas, nonnegative polynomials and sums of squares supported on circuits, {\em Res. Math. Sci.}, 3(1)(2016):9.

\bibitem{it}
I. Itenberg and M.F. Roy, Multivariate Descartes' Rule, IRMAR, Universit'e de Rennes I, 1996:337-346.

\bibitem{jo}
B. Joshi and A. Shiu, A Survey of Methods for Deciding Whether a Reaction Network is Multistationary, {\em Math. Model. Nat. Phenom.}, 10(5)(2015):47-67.

\bibitem{mu}
S. M\"uller, E. Feliu, et al., Sign Conditions for Injectivity of Generalized Polynomial Maps with Applications to Chemical Reaction Networks and Real Algebraic Geometry, {\em Found. Comput. Math.}, 16(2016):69-97.

\bibitem{la}
J.C. Lagarias and T.J. Richardson, Multivariate Descartes' rule of signs and Sturmfels's challenge problem, {\em Math. Intelligencer}, 19(3)(1997):9-15.

\bibitem{li}
T.Y. Li and X. Wang, On multivariate Descartes' rule - a counterexample, {\em Beitr. Algebra Geom.}, 39(1)(1998):1-5.

\bibitem{nie}
J. Nie, J. Demmel and B. Sturmfels, Minimizing polynomials via sum of squares over the gradient ideal, {\em Math. Program.}, 106(2006):587-606.

\bibitem{re}
B. Reznick, Extremal PSD forms with few terms, {\em Duke Math. J.}, 45(1978):363-374.

\bibitem{sc}
M. Schweighofer, Global Optimization of Polynomials Using Gradient Tentacles and Sums of Squares, {\em SIAM J. Optim.}, 17(3)(2006):920-942.

\bibitem{so}
F. Sottile, {\em Real solutions to equations from geometry}, AMS, University Lecture Series(57), 2011.

\bibitem{wang}
J. Wang, Nonnegative Polynomials and Circuit Polynomials, arXiv:1804.09455, 2018.

\end{thebibliography}

\end{document}